\documentclass[11pt]{article}
\usepackage[a4paper]{anysize}\marginsize{3.5cm}{3.5cm}{1.3cm}{2cm}
\pdfpagewidth=\paperwidth \pdfpageheight=\paperheight
\usepackage{amsfonts,amssymb,amsthm,amsmath,eucal}
\usepackage{pgf}
\usepackage{bbm,array}
\usepackage{tikz} %Used for drawing picture in LaTeX
\usepackage{float}
\restylefloat{table}
\usepackage{subfigure}
\usepackage{caption}
\usepackage{enumerate}

\usetikzlibrary{arrows}
\theoremstyle{plain}
\newtheorem{thm}{Theorem}[section]
\newtheorem{theorem}[thm]{Theorem}

\newtheorem{lemma}[thm]{Lemma}

\newtheorem{corollary}[thm]{Corollary}

\theoremstyle{definition}

\newtheorem{remark}[thm]{Remark}
\newtheorem{example}[thm]{Example}

\newtheorem{thevarthm}[thm]{\varthmname}

\newenvironment{varthm*}[1]{\trivlist\item[]{\bf #1.}\it}{\endtrivlist}

\renewcommand\geq{\geqslant}

\newcommand\be{\begin{eqnarray*}}
\newcommand\ee{\end{eqnarray*}}

\newcommand\newop[2]{\def#1{\mathop{\rm #2}\nolimits}}
\newop\log{log}
\newop\ord{ord}
\newop\Gal{Gal}
\newop\SL{SL}
\newop\GL{GL}
\newop\Bl{Bl}
\newop\mult{mult}
\newop\mass{mass}
\newop\div{div}
\newop\codim{codim}
\newop\sing{sing}
\newop\vdim{vdim}
\newop\edim{edim}
\newop\Ass{Ass}
\newop\size{size}
\newop\reg{reg}
\newop\areg{areg}
\newop\asreg{asreg}
\newop\satdeg{satdeg}
\newop\supp{supp}
\newop\gin{gin}
\newop\ini{in}
\newop\vol{vol}
\newop\sat{sat}
\newop\length{length}
\newop\depth{depth}
\newop\characteristic{char}
\newcommand\eqnref[1]{(\ref{#1})}

\def\keywordname{{\bfseries Keywords}}%
\def\keywords#1{\par\addvspace\medskipamount{\rightskip=0pt plus1cm
\def\and{\ifhmode\unskip\nobreak\fi\ $\cdot$
}\noindent\keywordname\enspace\ignorespaces#1\par}}
\def\subclassname{{\bfseries Mathematics Subject Classification
(2000)}\enspace}
\def\subclass#1{\par\addvspace\medskipamount{\rightskip=0pt plus1cm
\def\and{\ifhmode\unskip\nobreak\fi\ $\cdot$
}\noindent\subclassname\ignorespaces#1\par}}

\definecolor{qqqqff}{rgb}{0,0,0}
\definecolor{uuuuuu}{rgb}{0,0,0}
\definecolor{zzttqq}{rgb}{0,0,0}
\definecolor{xdxdff}{rgb}{0,0,0}
\definecolor{ttqqqq}{rgb}{0.,0.,0.}
\definecolor{wwccqq}{rgb}{0.,0.,0.}
\definecolor{qqqqcc}{rgb}{0.,0.,0.}
\definecolor{ffttww}{rgb}{0,0.,0.}
\begin{document}

\author{ M.~Lampa-Baczy\'nska\footnote{ML-B was partially supported by National Science Centre, Poland, grant 2016/23/N/ST1/01363},  D.~W\'ojcik }
\title{On the Pappus arrangement of lines, forth and back and to the point}
\date{\today}
\maketitle
\thispagestyle{empty}

\begin{abstract}

The purpose of this paper is to study the famous Pappus configuration of $9$ lines and its dual arrangement. We show among others that by applying the Pappus Theorem to the dual arrangement we obtain the configuration corresponding to the initial data of beginning configuration. We consider also the Pappus arrangements with some additional incidences and we establish algebraic conditions paralleling with these incidences.


\subclass{52C35, 32S22, 14N20, 13F20}
  
\end{abstract}

%*****************************************************************************

\section{Introduction}
\label{intro}

The Pappus theorem, proved in the fourth century, is one of the first theorems in what is now known as projective geometry.

\begin{theorem}[Pappus Theorem]\label{hexagon}
	
	Let $A_1$, $A_2$, $A_3$ and  $B_1$, $B_2$, $B_3$ be two triples of mutually distinct points lying on two distinct lines $L_{A}$ and $L_{B}$. Then the intersection points of pairs of lines $C_1=L(A_2,B_3) \cap L(A_3,B_2)$, $C_2=L(A_1,B_3) \cap L(A_3,B_1)$ and $C_3=L(A_1,B_2) \cap L(A_2,B_1)$ are collinear (see Figure \ref{pappus}).
\end{theorem}

Note that the labeling of points is taken so, that $A_{i}$, $B_{j}$, $C_{k}$ are collinear if and only, if $i+j+k\equiv0\ (mod\ 3)$.

There are numerous proofs of Pappus Theorem (see \cite{Gebert}, p.11-38). The one best appealing to algebraic geometers is probably that based on the Cayley-Bacharach Theorem. We recall this theorem below in the formulation attributed to Chasles, see \cite[Theorem CB3]{EGH}.

\begin{theorem}[Chasles]\label{thm: Chasles}
    Let $\Gamma_1$, $\Gamma_2$ be plane cubic curves meeting in nine points $P_1, \dots, P_9$. If $\Gamma$ is another cubic curve containing $P_1, \dots, P_8$ then $\Gamma$ contains also $P_9$.
\end{theorem}
    
This statement is applied to prove Pappus Theorem in the following way. Let $\Gamma_1$ be the union of lines $L(A_1,B_3) \cup L(A_2,B_1) \cup L(A_3,B_2)$ and let $\Gamma_2$ be the union of lines $L(A_1,B_2) \cup L(A_2,B_3) \cup L(A_3,B_1)$.
Then by construction

$$\Gamma_1 \cap \Gamma_2=\{A_1, A_2, A_3, B_1, B_2, B_3, C_1, C_2, C_3  \}.$$
Let $\Gamma$ be the union of lines $L_{A},L_{B}$ and $L(C_1, C_2)$. Theorem \ref{thm: Chasles} implies that it must be $C_3 \in \Gamma$. This implies in turn that the points $C_1, C_2, C_3$ are collinear as asserted. We denote this extra line by $L_{C}$.

\begin{figure}[ht]
	\centering
\definecolor{ffqqqq}{rgb}{1.,0.,0.}
\definecolor{uuuuuu}{rgb}{0.26666666666666666,0.26666666666666666,0.26666666666666666}
\definecolor{ccqqww}{rgb}{0.8,0.,0.4}
\begin{tikzpicture}[line cap=round,line join=round,>=triangle 45,x=0.7cm,y=0.7cm]
\clip(1.6609484780183619,-6.0384981054507705) rectangle (19.756186573256493,4.645321975795294);
\draw [line width=1.2pt, domain=3.1:16.756186573256493] plot(\x,{(-33.176428401696--2.9610389610389563*\x)/7.307359307359313});
\draw [line width=1.2pt,domain=5:13] plot(\x,{(-37.840752275046555--0.7965367965367953*\x)/6.025974025974031});
\draw [line width=0.6pt, dash pattern=on 2pt off 2pt, domain=3.2:10] plot(\x,{(-3.953751367589632-2.265849082223843*\x)/4.8850860133812795});
\draw [line width=0.6pt, dash pattern=on 2pt off 2pt ,domain=11.5:12.25] plot(\x,{(--56.134911855571836-4.8311688311688235*\x)/0.5714285714285712});
\draw [line width=0.6pt, dash pattern=on 2pt off 2pt, domain=5.5:17] plot(\x,{(--95.8387752792279-7.228210967212019*\x)/-9.404330561397519});
\draw [line width=0.6pt, dash pattern=on 2pt off 2pt,domain=5.5:12.5] plot(\x,{(--64.40800235516443-5.627705627705619*\x)/-5.4545454545454595});
\draw [line width=0.6pt, dash pattern=on 2pt off 2pt,domain=8.5:16.5] plot(\x,{(--94.94577781010634-6.827393382769199*\x)/-6.372058400830094});
\draw [line width=0.6pt, dash pattern=on 2pt off 2pt,domain=3.5:13] plot(\x,{(-14.036625333440707-1.8701298701298672*\x)/7.878787878787884});
\draw [line width=1.2pt,color=ffqqqq,domain=5.5:13.5] plot(\x,{(--26.54244430342639-1.590732564815187*\x)/-3.517123234995765});
\begin{scriptsize}
\draw [fill=ccqqww] (4.292983110052994,-2.8005654700921374) circle (2.5pt);
\draw[color=ccqqww] (4.414195231265121,-2.2) node {$A_1$};
\draw [fill=ccqqww] (11.600342417412307,0.16047349094681898) circle (2.5pt);
\draw[color=ccqqww] (11.721554538624442,0.7) node {$A_2$};
\draw [fill=ccqqww] (6.145796962866847,-5.4672321367588) circle (2.5pt);
\draw[color=ccqqww] (6.397009084078977,-5.8) node {$B_1$};
\draw [fill=ccqqww] (12.171770988840878,-4.670695340222005) circle (2.5pt);
\draw[color=ccqqww] (12.5,-5.2) node {$B_3$};
\draw [fill=ccqqww] (15.550127524264367,1.7609788304532188) circle (2.5pt);
\draw[color=ccqqww] (15.669606486676399,2.2) node {$A_3$};
\draw [fill=ccqqww] (9.178069123434273,-5.0664145523159805) circle (2.5pt);
\draw[color=ccqqww] (9.297312114382013,-5.6) node {$B_2$};
\draw [fill=uuuuuu] (7.354211785419602,-4.220454938886914) circle (3.0pt);
\draw[color=uuuuuu] (7.522420339490235,-4.8) node {$C_3$};
\draw [fill=uuuuuu] (11.876588775451257,-2.175063899746113) circle (3.0pt);
\draw[color=uuuuuu] (12.4,-2.5) node {$C_1$};
\draw [fill=uuuuuu] (8.359465540455492,-3.7657964645613) circle (3.0pt);
\draw[color=uuuuuu] (8.52674934381924,-4.3) node {$C_2$};
\end{scriptsize}
\end{tikzpicture}
	\caption{}
	\label{pappus}
\end{figure}

\section{A $(9_3)$ arrangement}

The configurations of $n$ points and $n$ lines with $k$ points per line and $k$ lines through each point are called in combinatorics as $(n_{k})$ arrangements. The Pappus configuration is a $(9_3)$ arrangement. 

Consider the incidence matrix of the Pappus arrangement.

\begin{table}[ht] $$
\begin{tabular}{|c|c|c|c|c|c|c|c|c|c|}
  \hline
    & $A_1$ & $A_2$ & $A_3$ & $B_1$ & $B_2$ & $B_3$ & $C_1$ & $C_2$ & $C_3$ \\
  \hline
  $L_{A}$ & + & + & + &  &  &  &  &  &   \\
  \hline
  $L_{B}$ &  &  &  & + & + &  +&  &  &    \\
  \hline
  $L_{C}$ &  &  &  &  &  &  & + & + & +  \\
  \hline
  $L(A_2,B_3)$ &  & + &  &  &  & + & + &  &   \\
  \hline
  $L(A_3,B_2)$ &  &  & + &  & + &  & + &  &  \\
  \hline
  $L(A_1,B_3)$ & + &  &  &  &  & + &  & + &   \\
  \hline
  $L(A_3,B_1)$ &  &  & + & + &  &  &  & + &   \\
  \hline
 $L(A_1,B_2)$ & + &  &  &  & + &  &  &  &  + \\
  \hline
  $L(A_2,B_1)$ &  & + &  & + &  &  & &  & +  \\

  \hline
\end{tabular}$$
\caption{: Incidence matrix of the Pappus arrangement}
\label{tab:config iv}
\end{table} 

The cross "+" in the row $L$ and column $X$ indicates that the line $L$ contains the point $X$. There are exactly $3$ of the distinguished points on each of the arrangement lines, and vice versa, each of distinguished points is contained in exactly $3$ of the arrangement lines.
Such an arrangement is called self-dual. Besides combinatorics, this name is justified by the following obvious observation.

For a point $X \in \mathbb{P}^2$, let $X^{*}$ denote the dual line in the dual projective space $(\mathbb{P}^2)^{V}$. Specifically, if $X=(a:b:c)$ and the coordinates in  $(\mathbb{P}^2)^{V}$ are $A,B,C$, then
$$X^{*}= \{ (A:B:C) \in  (\mathbb{P}^2)^{V} : aA+bB+cC=0 \} .$$
Similarly, if $Ax+By +Cz=0$ is a line $L \subset \mathbb{P}^2$, then $L^{*}$ is the dual point in  $(\mathbb{P}^2)^{V}$ with coordinates 
$$L^{*}= (A:B:C).$$
We have an immediate corollary.

\begin{theorem}\textbf{(dual Pappus statement)}
    Keeping the notation from Theorem \ref{hexagon}, the lines $A_1^{*}, A_2^{*}, A_3^{*}, B_1^{*}, B_2^{*}, B_3^{*}, C_1^{*}, C_2^{*}, C_3^{*}$ and points $L_{A}^{*}, L_{B}^{*}, L_{C}^{*}, L^{*}(A_2,B_3),$  $L^{*}(A_3,B_2), L^{*}(A_1,B_3), L^{*}(A_3,B_1),L^{*}(A_1,B_2), L^{*}(A_2,B_1) $ form a Pappus $(9_3)$ arrangement.
\end{theorem}

\begin{remark}
It is well-known that there are combinatorially three distinct $(9_3)$ arrangements. It was firstly shown by Kantor in \cite{kantor}. One of them is being that of Pappus. As our focus is in geometry, we do not dwell on this aspect here.
\end{remark}

\section{Pappus lines}
Passing to the dual arrangement, it is easy to see that there is certain ambiguity in deciding which dual lines play the role of apparently distinguishes, lines $L_{A}$ and $L_{B}$. It turns out however that these lines are not special at all. We illustrate it with two examples.

\begin{example}
Consider the configuration of points as in Figure \ref{pappus} but let now relabel the points $C_{i}$ by $A_{i}$ (see Figure \ref{relabel}).

\begin{figure}[ht]
	\centering
\definecolor{ffqqqq}{rgb}{1.,0.,0.}
\definecolor{uuuuuu}{rgb}{0.26666666666666666,0.26666666666666666,0.26666666666666666}
\definecolor{ccqqww}{rgb}{0.8,0.,0.4}
\begin{tikzpicture}[line cap=round,line join=round,>=triangle 45,x=0.7cm,y=0.7cm]
\clip(1.6609484780183619,-6.0384981054507705) rectangle (19.756186573256493,4.645321975795294);
\draw [line width=1.2pt, domain=3.1:16.756186573256493] plot(\x,{(-33.176428401696--2.9610389610389563*\x)/7.307359307359313});
\draw [line width=1.2pt,domain=5:13] plot(\x,{(-37.840752275046555--0.7965367965367953*\x)/6.025974025974031});
\draw [line width=0.6pt, dash pattern=on 2pt off 2pt, domain=3.2:10] plot(\x,{(-3.953751367589632-2.265849082223843*\x)/4.8850860133812795});
\draw [line width=0.6pt, dash pattern=on 2pt off 2pt ,domain=11.5:12.25] plot(\x,{(--56.134911855571836-4.8311688311688235*\x)/0.5714285714285712});
\draw [line width=0.6pt, dash pattern=on 2pt off 2pt, domain=5.5:17] plot(\x,{(--95.8387752792279-7.228210967212019*\x)/-9.404330561397519});
\draw [line width=0.6pt, dash pattern=on 2pt off 2pt,domain=5.5:12.5] plot(\x,{(--64.40800235516443-5.627705627705619*\x)/-5.4545454545454595});
\draw [line width=0.6pt, dash pattern=on 2pt off 2pt,domain=8.5:16.5] plot(\x,{(--94.94577781010634-6.827393382769199*\x)/-6.372058400830094});
\draw [line width=0.6pt, dash pattern=on 2pt off 2pt,domain=3.5:13] plot(\x,{(-14.036625333440707-1.8701298701298672*\x)/7.878787878787884});
\draw [line width=1.2pt,color=ffqqqq,domain=5.5:13.5] plot(\x,{(--26.54244430342639-1.590732564815187*\x)/-3.517123234995765});
\begin{scriptsize}
\draw [fill=ccqqww] (4.292983110052994,-2.8005654700921374) circle (2.5pt);
\draw[color=ccqqww] (4.414195231265121,-2.2) node {$C_1$};
\draw [fill=ccqqww] (11.600342417412307,0.16047349094681898) circle (2.5pt);
\draw[color=ccqqww] (11.721554538624442,0.7) node {$C_2$};
\draw [fill=ccqqww] (6.145796962866847,-5.4672321367588) circle (2.5pt);
\draw[color=ccqqww] (6.397009084078977,-5.8) node {$B_1$};
\draw [fill=ccqqww] (12.171770988840878,-4.670695340222005) circle (2.5pt);
\draw[color=ccqqww] (12.5,-5.2) node {$B_3$};
\draw [fill=ccqqww] (15.550127524264367,1.7609788304532188) circle (2.5pt);
\draw[color=ccqqww] (15.669606486676399,2.2) node {$C_3$};
\draw [fill=ccqqww] (9.178069123434273,-5.0664145523159805) circle (2.5pt);
\draw[color=ccqqww] (9.297312114382013,-5.6) node {$B_2$};
\draw [fill=uuuuuu] (7.354211785419602,-4.220454938886914) circle (3.0pt);
\draw[color=uuuuuu] (7.522420339490235,-4.8) node {$A_3$};
\draw [fill=uuuuuu] (11.876588775451257,-2.175063899746113) circle (3.0pt);
\draw[color=uuuuuu] (12.4,-2.5) node {$A_1$};
\draw [fill=uuuuuu] (8.359465540455492,-3.7657964645613) circle (3.0pt);
\draw[color=uuuuuu] (8.52674934381924,-4.3) node {$A_2$};
\end{scriptsize}
\end{tikzpicture}
	\caption{}
	\label{relabel}
\end{figure}
Let as before
\begin{displaymath}
\begin{array}{ccl}
\vspace{0,5cm}
C_1 & = & L(A_2,B_3) \cap L(A_3,B_2),\\
\vspace{0,5cm}

C_2 & = & L(A_1,B_3) \cap L(A_3,B_1),\\
C_3 & = & L(A_1,B_2) \cap L(A_2,B_1). \\
\end{array}
\end{displaymath}  
Then we obtain exactly the same picture as before, with changed names of lines $L_{A}$ and $L_{C}$.

It may come as a surprise that one can start with arbitrary two lines (not intersecting in an arrangement point), choosing the labeling of points carefully. Consider for example the following relabeling:
\begin{displaymath}
\begin{array}{ccc}
\vspace{0,5cm}
B_1 \to A_1, & C_2 \to A_2, & A_3 \to A_3,\\
\vspace{0,5cm}

C_1 \to B_1, & A_2 \to B_2, & B_3 \to B_3.\\
\end{array}
\end{displaymath}  
Thus  the "old" line $L(B_1, A_3)$ becomes the "new" line $L_{A}$ and the "old" line $L(A_2, B_3)$ becomes the "new" line $L_{B}$. Moreover we have:
\begin{displaymath}
\begin{array}{ccl}
\vspace{0,5cm}
C_3 \to C_3& = & L(A_1,B_2) \cap L(A_2,B_1),\\
\vspace{0,5cm}

A_2 \to C_2 & = & L(A_2,B_3) \cap L(A_3,B_2),\\
B_2 \to C_1 & = & L(A_1,B_3) \cap L(A_3,B_1), \\
\end{array}
\end{displaymath} 
where the lines are taken with respect to the "new" labeling. Thus the "new" line $L_{C}$ is the "old" line $L(A_1, B_2)$. The figure remains unchanged.
\end{example}

It is natural to wonder what happens, if we are not careful labeling the points. We consider this question in the original picture. Let $\sigma=(a\, b\, c)$ be a permutation of the set $ \{ 1,2,3 \}$. By this notation we mean that $\sigma$ assigns $a$ to $1$ and similarly $b$ to $2$ and $c$ to $3$. Consider the initial figure with labeling as in Figure \ref{permutation}.

Thus the labeling in Figure \ref{pappus} corresponds to $\sigma = id$.

\begin{figure}[h]
	\centering
\definecolor{ffqqqq}{rgb}{1.,0.,0.}
\definecolor{uuuuuu}{rgb}{0.26666666666666666,0.26666666666666666,0.26666666666666666}
\definecolor{ccqqww}{rgb}{0.8,0.,0.4}
\begin{tikzpicture}[line cap=round,line join=round,>=triangle 45,x=0.7cm,y=0.7cm]
\clip(1.6609484780183619,-6.0384981054507705) rectangle (19.756186573256493,3.645321975795294);
\draw [line width=1.2pt, domain=3.1:16.756186573256493] plot(\x,{(-33.176428401696--2.9610389610389563*\x)/7.307359307359313});
\draw [line width=1.2pt,domain=5:13] plot(\x,{(-37.840752275046555--0.7965367965367953*\x)/6.025974025974031});
%\draw [line width=0.6pt, dash pattern=on 2pt off 2pt, domain=3.2:10] plot(\x,{(-3.953751367589632-2.265849082223843*\x)/4.8850860133812795});
%\draw [line width=0.6pt, dash pattern=on 2pt off 2pt ,domain=11.5:12.25] plot(\x,{(--56.134911855571836-4.8311688311688235*\x)/0.5714285714285712});
%\draw [line width=0.6pt, dash pattern=on 2pt off 2pt, domain=5.5:17] plot(\x,{(--95.8387752792279-7.228210967212019*\x)/-9.404330561397519});
%\draw [line width=0.6pt, dash pattern=on 2pt off 2pt,domain=5.5:12.5] plot(\x,{(--64.40800235516443-5.627705627705619*\x)/-5.4545454545454595});
%\draw [line width=0.6pt, dash pattern=on 2pt off 2pt,domain=8.5:16.5] plot(\x,{(--94.94577781010634-6.827393382769199*\x)/-6.372058400830094});
%\draw [line width=0.6pt, dash pattern=on 2pt off 2pt,domain=3.5:13] plot(\x,{(-14.036625333440707-1.8701298701298672*\x)/7.878787878787884});
%\draw [line width=1.2pt,color=ffqqqq,domain=5.5:13.5] plot(\x,{(--26.54244430342639-1.590732564815187*\x)/-3.517123234995765});
\begin{scriptsize}
\draw [fill=ccqqww] (4.292983110052994,-2.8005654700921374) circle (2.5pt);
\draw[color=ccqqww] (4.414195231265121,-2.2) node {$A_1$};
\draw [fill=ccqqww] (11.600342417412307,0.16047349094681898) circle (2.5pt);
\draw[color=ccqqww] (11.721554538624442,0.7) node {$A_2$};
\draw [fill=ccqqww] (6.145796962866847,-5.4672321367588) circle (2.5pt);
\draw[color=ccqqww] (6.397009084078977,-5.9) node {$B_{\sigma(1)}$};
\draw [fill=ccqqww] (12.171770988840878,-4.670695340222005) circle (2.5pt);
\draw[color=ccqqww] (12.5,-5.2) node {$B_{\sigma(3)}$};
\draw [fill=ccqqww] (15.550127524264367,1.7609788304532188) circle (2.5pt);
\draw[color=ccqqww] (15.669606486676399,2.2) node {$A_3$};
\draw [fill=ccqqww] (9.178069123434273,-5.0664145523159805) circle (2.5pt);
\draw[color=ccqqww] (9.297312114382013,-5.6) node {$B_{\sigma(2)}$};
%\draw [fill=uuuuuu] (7.354211785419602,-4.220454938886914) circle (3.0pt);
%\draw[color=uuuuuu] (7.522420339490235,-4.8) node {$C_3$};
%\draw [fill=uuuuuu] (11.876588775451257,-2.175063899746113) circle (3.0pt);
%\draw[color=uuuuuu] (12.4,-2.5) node {$C_1$};
%\draw [fill=uuuuuu] (8.359465540455492,-3.7657964645613) circle (3.0pt);
%\draw[color=uuuuuu] (8.52674934381924,-4.3) node {$C_2$};
\end{scriptsize}
\end{tikzpicture}
	\caption{}
	\label{permutation}
\end{figure}

\begin{example}
We consider now Figure  \ref{pappus} with the labeling on line $L_{B}$ changed by the permutation $\sigma = (2 \ 1 \ 3)$. We obtain then a new Pappus line $L_{C, \sigma}$ (see Figure \ref{sigma}).

\begin{figure}[h]
	\centering
\definecolor{ffqqqq}{rgb}{1.0,0.0,0.0}
\definecolor{uuuuuu}{rgb}{0.26666666666666666,0.26666666666666666,0.26666666666666666}
\definecolor{ccqqww}{rgb}{0.8,0.0,0.4}
\begin{tikzpicture}[line cap=round,line join=round,>=triangle 45,x=0.7cm,y=0.7cm]
\clip(3.0400000000000007,-10.08280000000000003) rectangle (16.680000000000017,3.58);
\draw [line width=1.2000000000000002pt,domain=3.3400000000000007:16.680000000000017] plot(\x,{(-51.0904--4.56*\x)/11.260000000000002});
\draw [line width=1.2000000000000002pt,domain=4.6400000000000007:14.680000000000017] plot(\x,{(-37.84939999999999--0.7999999999999998*\x)/6.02});
\draw [dash pattern=on 2pt off 2pt,domain=3.5400000000000007:8.680000000000017] plot(\x,{(-6.2463000000000015--2.67*\x)/-1.8600000000000003});
\draw [dash pattern=on 2pt off 2pt,domain=7.063400000000000007:12.380000000000017] plot(\x,{(--60.280800000000006-5.23*\x)/-2.42});
\draw [dash pattern=on 2pt off 2pt,domain=8.3400000000000007:16.680000000000017] plot(\x,{(--94.99530000000001-6.83*\x)/-6.370000000000001});
\draw [dash pattern=on 2pt off 2pt,domain=3.3400000000000007:13.680000000000017] plot(\x,{(-14.041700000000002-1.87*\x)/7.88});
\draw [dash pattern=on 2pt off 2pt,domain=11.4400000000000007:12.3680000000000017] plot(\x,{(--56.1192-4.83*\x)/0.5700000000000003});
\draw [dash pattern=on 2pt off 2pt,domain=5.3400000000000007:16.680000000000017] plot(\x,{(--95.8825-7.2299999999999995*\x)/-9.4});
\draw [line width=1.2000000000000002pt,color=ffqqqq,domain=7.3400000000000007:12.580000000000017] plot(\x,{(-46.748034563980426--3.7623830225534087*\x)/2.167839828046496});
\begin{scriptsize}
\draw [fill=ccqqww] (4.29,-2.8) circle (2.5pt);
\draw[color=ccqqww] (4.320000000000003,-2.420000000000002) node {$A_1$};
\draw [fill=ccqqww] (11.6,0.16) circle (2.5pt);
\draw[color=ccqqww] (11.174000000000001,0.43999999999999906) node {$A_2$};
\draw [fill=ccqqww] (15.55,1.76) circle (2.5pt);
\draw[color=ccqqww] (15.2680000000000016,2.0399999999999996) node {$A_3$};
\draw [fill=ccqqww] (9.18,-5.07) circle (2.5pt);
\draw[color=ccqqww] (9.032000000000001,-4.780000000000003) node {$B_1$};
\draw [fill=ccqqww] (6.15,-5.47) circle (1.5pt);
\draw[color=ccqqww] (6.280000000000005,-5.180000000000003) node {$B_2$};
\draw [fill=ccqqww] (12.17,-4.67) circle (2.5pt);
\draw[color=ccqqww] (12.400000000000011,-4.280000000000003) node {$B_3$};
\draw [fill=uuuuuu] (7.859457786672598,-7.923899080868727) circle (3.0pt);
\draw[color=uuuuuu] (7.2040000000000006,-7.7400000000000045) node {$C_3$};
\draw [fill=uuuuuu] (10.027297614719094,-4.161516058315319) circle (3.0pt);
\draw[color=uuuuuu] (10.200000000000008,-4.62800000000000025) node {$C_2$};
\draw [fill=uuuuuu] (11.755594964774016,-1.158462596242965) circle (3.0pt);
\draw[color=uuuuuu] (12.294000000000001,-1.1800000000000015) node {$C_1$};
\end{scriptsize}
\end{tikzpicture}
	\caption{}
	\label{sigma}
\end{figure}
\end{example}

If the initial data, i.e., the points $A_1, A_2, A_3, B_1, B_2, B_3$ is sufficiently general, then there are six distinct Pappus lines indexed by permutations $\sigma \in S_3$. The following fact is, probably, well-known but we were not able to trace it down in the literature.

\begin{theorem}\label{dualthm}
    The dual points $L_{C, \sigma}^{*}$ with $\sigma \in S_3$ lie by three on a pair of lines in $(\mathbb{P}^2)^{V}$. The Pappus lines determined by these six points correspond to new triples of points  $A'_1, A'_2, A'_3$ and $ B'_1, B'_2, B'_3$ on the initial lines $L_A$ and $L_B$.
\end{theorem}
\begin{proof}
Our approach is algebraic and motivated by Tao's survey \cite{Tao2014}. Without loss of generality we may assume that
$$A_1=(1:0:0), \hspace{0.3cm} A_3=(0:1:0), \hspace{0.3cm} B_2=(0:0:1), \hspace{0.3cm} B_3=(1:1:1).$$
Thus $L_{A}: z=0 $ and $L_{B}:x-y=0$. For the remaining points we introduce parameters $a,b \in \mathbb{C} \setminus \{0,1\}$:
$$A_2=(1:a:0) \hspace{0.3cm} \textrm{and} \hspace{0.3cm} B_1= (1:1:b).$$
These assumptions are sufficient to determine all relevant lines $L(A_{i}, B_{j})$ for $i,j \in \{1,2,3\}$:
\begin{figure}[ht]
	\centering 
	\begin{minipage}[b]{0.47\linewidth}
	$$\begin{array}{ccl}
\vspace{0,5cm}
L(A_1, B_1) & : & by-z=0,\\ 
\vspace{0,5cm}
L(A_1, B_2) & : &  y= 0, \\
\vspace{0,5cm}
L(A_1,B_3) & : &  y-z =0, \\
\vspace{0,5cm}
L(A_3,B_1) & : & bx-z=0,\\ 

L(A_3,B_2) & : &  x=0.\\
\end{array}$$
	\end{minipage}
	\quad
	\begin{minipage}[b]{0.47\linewidth}
$$\begin{array}{ccl}
\vspace{0,5cm}
L(A_2,B_1) & : & abx-by+(1-a)z=0,\\ 
\vspace{0,5cm}
L(A_2,B_2) & : &  ax-y=0,\\
\vspace{0,5cm}
L(A_2,B_3) & : &  ax-y+(1-a)z=0,\\
\vspace{0,5cm}
L(A_3,B_3) & : & x-z=0,\\ 
 & & \\
 \end{array}$$
	\end{minipage}
\end{figure}

\hspace{0.2cm}

Their intersection points are:
\begin{displaymath}
\begin{array}{ccl}
\vspace{0,5cm}
C_1, \sigma & = & L(A_2,B_{\sigma(3)}) \cap L(A_3,B_{\sigma(2)}),\\
\vspace{0,5cm}

C_2, \sigma & = & L(A_1,B_{\sigma(3)}) \cap L(A_3,B_{\sigma(1)}),\\
C_3, \sigma & = & L(A_1,B_{\sigma(2)}) \cap L(A_2,B_{\sigma(1)}), \\
\end{array}
\end{displaymath} 
with $\sigma \in S_3$. Let us notice that the permutations $\tau_2 =(2 \ 1 \ 3)$ and $\tau_3 =(3 \ 1 \ 2)$ generate the symmetric group $S_3$. We have the following coordinates of points $C_{k, \sigma}$ for $\sigma \in S_3$ and $k \in \{1,2,3 \}$:

\begin{figure}[ht]
	\centering
	\begin{minipage}[b]{0.47\linewidth}
$$\begin{array}{ccl}
\vspace{0,5cm}
C_{1,id} & = &  (0:1-a:1),\\
\vspace{0,5cm}
C_{2,id} & = &  (1:b:b) ,\\
\vspace{1cm}
C_{3, id} & = & (a-1:0:ab), \\
\vspace{0,5cm}
C_{1,  \tau_3} & = &  (1:a:b),\\
\vspace{0,5cm}
C_{2, \tau_3} & = &  (1:0:1), \\
\vspace{1cm}
C_{3,  \tau_3} & = &  (ab-b+1:a:ab), \\
\vspace{0.5cm}
C_{1,\tau_3^2} & = &  (b:ab-a+1:b),\\
\vspace{0,5cm}
C_{2,\tau_3^2} & = &  (0:1:b) ,\\
C_{3, \tau_3^2} & = &  (1:a:a), \\
\end{array}$$
	\end{minipage}
	\quad
	\begin{minipage}[b]{0.47\linewidth}
$$\begin{array}{ccl}
\vspace{0,5cm}
C_{1,\tau_2} & = & (1:a+b-ab:b),\\
\vspace{0,5cm}
C_{2,\tau_2} & = & (0:1:1), \\
\vspace{1cm}
C_{3,\tau_2} & = & (1:a:ab),\\
\vspace{0,5cm}
C_{1,\tau_2 \tau_3} & = &  (0:1-a:b), \\
\vspace{0.5cm}
C_{2,\tau_2 \tau_3} & = &  (b:1:b) ,\\
\vspace{1cm}
C_{3, \tau_2 \tau_3 } & = &  (a-1:0:a),\\
\vspace{0,5cm}
C_{1,\tau_2 \tau_3^2} & = &  (1:a:1),\\
\vspace{0,5cm}
C_{2,\tau_2 \tau_3^2} & = &  (1:0:b) ,\\
C_{3, \tau_2 \tau_3^2} & = &  (b+a-1:ab:ab). \\
\end{array}$$
	\end{minipage}
\end{figure}

Thus we have an arrangement of $9$ lines $L(A_{i}, B_{j})$ and altogether $24$ points where they intersect. In six of these $24$ points three arrangement lines meet (they are $A_1, A_2, A_3, B_1, B_2, B_3$) and the remaining points $C_{i, \sigma}$ are intersection points of exactly two lines. 

For any collection of $n$ distinct lines in $\mathbb{P}^2$ we have the following formula concerning the number of intersection points with their multiplicities

\begin{equation} \label{multiple}
    \binom{n}{2}= {\sum}_{k\geq 2} t_{k} \cdot \binom{k}{2},
\end{equation}

\vspace{0.3cm}
\noindent where $t_k$ is the number of points, where exactly $k$ lines meet. By (\ref{multiple}) we conclude that there does not occur  any additional incidences between these $9$ lines. But there are additional collinearities of points $C_{1,\sigma}, C_{2,\sigma}, C_{3,\sigma}$ for all $\sigma \in S_3$:

$$\begin{array}{lcl}
\vspace{0,5cm}
L_{C, id} & : & abx-y-(a-1)z=0,\\
\vspace{0,5cm}
L_{C, \tau_2} & : & a(1-b)x-y+z=0,\\
\vspace{0,5cm}
L_{C, \tau_3} & : & ax+(b-1)y-az=0,\\
\vspace{0,5cm}
L_{C, \tau_2 \tau_3} & : & ax-by-(a-1)z=0, \\
\vspace{0,5cm}
L_{C, \tau_3^2}  & : & a(1-b)x+by-z=0,\\
L_{C, \tau_2 \tau_3^2} & : &  abx+(1-b)y -az=0. \\
\end{array}$$

\noindent
They determine six points $L_{C, \sigma}^{*}$ in $(\mathbb{P}^2)^{V}$:

\begin{figure}[ht]
	\centering
	\begin{minipage}[b]{0.47\linewidth}
$$\begin{array}{lcl}
\vspace{0,5cm}
L_{C, id}^{*} & = &  (ab:-1:1-a),\\
\vspace{0,5cm}
L_{C, \tau_3}^{*} & = &  (a:b-1:-a),\\
L_{C, \tau_3^2}^{*}  & = &  (a-ab:b:-1) , \\
\end{array}$$
	\end{minipage}
	\quad
	\begin{minipage}[b]{0.47\linewidth}
$$\begin{array}{lcl}
\vspace{0,5cm}
L_{C, \tau_2}^{*} & = &  (a-ab:-1:1) ,\\
\vspace{0,5cm}
L_{C, \tau_2 \tau_3}^{*} & = &  (a:-b:1-a),\\
L_{C, \tau_2 \tau_3^2}^{*} & = &  (ab:1-b:-a). \\
\end{array}$$
	\end{minipage}
\end{figure}

It is now easy to check that these points lie on two lines $M_1, M_2$:
\vspace{0,3cm}
$$\begin{array}{lcl}
\vspace{0,5cm}
M_1  & : &  (ab-b+1)\,x+a(ab-a+1)\,y+a(b^2-b+1)\,z=0,\\
M_2 & : &  (a+b-1)\,x+a(a+b-ab)\,y+a(b^2-b+1)\,z=0. \\
\end{array}$$
\begin{figure}[h]
	\centering
\definecolor{ffqqqq}{rgb}{1.,0.,0.}
\definecolor{uuuuuu}{rgb}{0.26666666666666666,0.26666666666666666,0.26666666666666666}
\definecolor{ccqqww}{rgb}{0.8,0.,0.4}
\begin{tikzpicture}[line cap=round,line join=round,>=triangle 45,x=0.7cm,y=0.7cm]
\clip(-0.06609484780183619,-5.8050384981054507705) rectangle (19.756186573256493,2.80645321975795294);
\draw [line width=1.2pt, domain=3.1:16.756186573256493] plot(\x,{(-33.176428401696--2.9610389610389563*\x)/7.307359307359313});
\draw[color=black] (2.54,-3.27) node {$M_1$};
\draw[color=black] (4.44,-5.57) node {$M_2$};
\draw [line width=1.2pt,domain=5:15] plot(\x,{(-37.840752275046555--0.7965367965367953*\x)/6.025974025974031});
%\draw [line width=0.6pt, dash pattern=on 2pt off 2pt, domain=3.2:10] plot(\x,{(-3.953751367589632-2.265849082223843*\x)/4.8850860133812795});
%\draw [line width=0.6pt, dash pattern=on 2pt off 2pt ,domain=11.5:12.25] plot(\x,{(--56.134911855571836-4.8311688311688235*\x)/0.5714285714285712});
%\draw [line width=0.6pt, dash pattern=on 2pt off 2pt, domain=5.5:17] plot(\x,{(--95.8387752792279-7.228210967212019*\x)/-9.404330561397519});
%\draw [line width=0.6pt, dash pattern=on 2pt off 2pt,domain=5.5:12.5] plot(\x,{(--64.40800235516443-5.627705627705619*\x)/-5.4545454545454595});
%\draw [line width=0.6pt, dash pattern=on 2pt off 2pt,domain=8.5:16.5] plot(\x,{(--94.94577781010634-6.827393382769199*\x)/-6.372058400830094});
%\draw [line width=0.6pt, dash pattern=on 2pt off 2pt,domain=3.5:13] plot(\x,{(-14.036625333440707-1.8701298701298672*\x)/7.878787878787884});
%\draw [line width=1.2pt,color=ffqqqq,domain=5.5:13.5] plot(\x,{(--26.54244430342639-1.590732564815187*\x)/-3.517123234995765});
\begin{scriptsize}
\draw [fill=ccqqww] (4.292983110052994,-2.8005654700921374) circle (2.5pt);
\draw[color=ccqqww] (4.414195231265121,-2.2) node {$L_{C, id}^{*}$};
\draw [fill=ccqqww] (11.600342417412307,0.16047349094681898) circle (2.5pt);
\draw[color=ccqqww] (11.721554538624442,0.7) node {$L_{C, \tau_3}^{*}$};
\draw [fill=ccqqww] (6.145796962866847,-5.4672321367588) circle (2.5pt);
\draw[color=ccqqww] (6.397009084078977,-5.0) node {$L_{C, \tau_2}^{*}$};
\draw [fill=ccqqww] (12.171770988840878,-4.670695340222005) circle (2.5pt);
\draw[color=ccqqww] (12.5,-4.1) node {$L_{C, \tau_2 \tau_3^2}^{*}$};
\draw [fill=ccqqww] (15.550127524264367,1.7609788304532188) circle (2.5pt);
\draw[color=ccqqww] (15.69606486676399,2.4) node {$L_{C, \tau_3^2}^{*}$};
\draw [fill=ccqqww] (9.178069123434273,-5.0664145523159805) circle (2.5pt);
\draw[color=ccqqww] (9.297312114382013,-4.5) node {$L_{C, \tau_2 \tau_3}^{*}$};
%\draw [fill=uuuuuu] (7.354211785419602,-4.220454938886914) circle (3.0pt);
%\draw[color=uuuuuu] (7.522420339490235,-4.8) node {$C_3$};
%\draw [fill=uuuuuu] (11.876588775451257,-2.175063899746113) circle (3.0pt);
%\draw[color=uuuuuu] (12.4,-2.5) node {$C_1$};
%\draw [fill=uuuuuu] (8.359465540455492,-3.7657964645613) circle (3.0pt);
%\draw[color=uuuuuu] (8.52674934381924,-4.3) node {$C_2$};
\end{scriptsize}
\end{tikzpicture}
	\caption{}
	\label{dual}
\end{figure}

The data in Figure \ref{dual} is exactly the initial data in the Pappus Theorem. This data determines six Pappus lines $M_{C,\sigma}$ with $\sigma \in S_3$. Since all calculations are parallel to those presented in the first part of this proof we omit them and conclude establishing the coordinates of their dual points:

\begin{figure}[ht]
	\centering
	\begin{minipage}[b]{0.47\linewidth}
$$\begin{array}{lcl}
\vspace{0,5cm}
M_{C, id}^{*} & = &  (b-2:b-2:2b^2 -2b-1),\\
\vspace{0,5cm}
M_{C, \tau_3}^{*} & = &(b+1:b+1:-b^2+4b-1),\\
M_{C, \tau_3^2}^{*}  & = &(2b-1:2b-1:b^2+2b-2), \\
\end{array}$$
	\end{minipage}
	\quad
	\begin{minipage}[b]{0.47\linewidth}
$$\begin{array}{lcl}
\vspace{0,5cm}
\hspace{0.5cm} M_{C, \tau_2}^{*} & = & (2a-1:a^2 -2a:0),\\
\vspace{0,5cm}
\hspace{0.5cm} M_{C, \tau_2 \tau_3}^{*} & = &(a+1:2a^2-a:0),\\
\hspace{0.5cm} M_{C, \tau_2 \tau_3^2}^{*} & = &(a-2:-a^2-a:0). \\
\end{array}$$
	\end{minipage}
\end{figure}
They lie on lines $L_A$ and $L_B$ as follows 
$$\begin{array}{lcl}
\vspace{0,5cm}
M_{C, id}^{*}, \hspace{0.3cm} M_{C, \tau_3}^{*}, \hspace{0.3cm} M_{C, \tau_3^2}^{*} & \in & L_B,\\
M_{C, \tau_2}^{*}, \hspace{0.3cm}  M_{C, \tau_2 \tau_3}^{*},  \hspace{0.3cm} M_{C, \tau_2 \tau_3^2}^{*}   & \in & L_A.\\
\end{array}$$
\end{proof}

\begin{corollary}
It follows immediately from the proof of Theorem \ref{dualthm} that the points dual to lines $L_{C, \sigma}$ determined by even permutations lie on the line $M_1$, whereas those corresponding to odd permutations are on the line  $M_2$.
\end{corollary}{}

\begin{remark}
Although the equations of lines in the arrangement depend on two parameters ($a$ and $b$) after the double dualisation the coordinates of points lying back on the line $L_A$ depend only on the parameter $a$, and respectively the coordinates of points lying back on the line $L_B$ depend only on the parameter $b$. It is a quite mysterious phenomenon for which we have no explanation at the moment.
\end{remark}

\section{Additional incidences}\label{sec: additional incidences}
In this section we study the question which initial data, i.e. points $A_1, A_2, A_3, B_1,$ $B_2, B_3$ lead to Pappus lines passing through the intersection point of the initial lines 
$$S=L_{A} \cap L_{B}=(1:1:0)$$ 
(as in Figure \ref{pekowe}). 
In particular we show that this additional incidence may happen only for at most two lines among the six lines $L_{C, \sigma}$ with $\sigma\in S_3$, under the assumption that all these six lines are mutually distinct.

\begin{figure}[ht]
	\centering
\definecolor{uuuuuu}{rgb}{0.26666666666666666,0.26666666666666666,0.26666666666666666}
\definecolor{ffqqqq}{rgb}{1.0,0.0,0.0}
\definecolor{ccqqqq}{rgb}{0.8,0.0,0.0}
\begin{tikzpicture}[line cap=round,line join=round,>=triangle 45,x=1.3cm,y=1.2cm]
\clip(-2.9537930699128174,-4.803663487474415) rectangle (5.083120208313234,4.572735337122672);
\draw [line width=1.2pt,domain=-2.9537930699128174:5.083120208313234] plot(\x,{(-0.0--4.0*\x)/2.0});
\draw [line width=1.2pt] (1.0,-3) -- (1.0,2.5);
\draw [line width=0.6pt,dash pattern=on 2pt off 2pt,domain=-0.5:4.5] plot(\x,{(-0.0-0.0*\x)/1.0});
\draw [line width=0.6pt,dash pattern=on 2pt off 2pt,domain=-2.2:4.5] plot(\x,{(--8.0-2.0*\x)/-3.0});
\draw [line width=0.6pt,dash pattern=on 2pt off 2pt,domain=0.8:2.5] plot(\x,{(--8.0-6.0*\x)/-1.0});
\draw [line width=0.6pt,dash pattern=on 2pt off 2pt,domain=-0.8:1.9] plot(\x,{(-0.0-1.0*\x)/-1.0});
\draw [line width=0.6pt,dash pattern=on 2pt off 2pt,domain=-2.9537930699128174:1.7] plot(\x,{(--2.0-5.0*\x)/-3.0});
\draw [line width=0.6pt,dash pattern=on 2pt off 2pt,domain=0.8:5.083120208313234] plot(\x,{(-4.0--4.0*\x)/1.0});
\draw [line width=1.2pt,color=ffqqqq,domain=0.5:4.6] plot(\x,{(-8.0--2.0*\x)/-3.0});
\begin{scriptsize}
\draw [fill=ccqqqq] (0.0,0.0) circle (2pt);
\draw[color=ccqqqq] (-0.3,0.18657678331955307) node {$B_{\sigma(1)}$};
\draw [fill=ccqqqq] (1.0,0.0) circle (2pt);
\draw[color=ccqqqq] (0.8,0.18657678331955307) node {$A_2$};
\draw [fill=ccqqqq] (1.0,1.0) circle (2pt);
\draw[color=ccqqqq] (0.8,1.1846248374783466) node {$A_3$};
\draw [fill=ccqqqq] (1.0,-2.0) circle (2pt);
\draw[color=ccqqqq] (0.8,-1.8226515362369655) node {$A_1$};
\draw [fill=ccqqqq] (-2.0,-4.0) circle (2pt);
\draw[color=ccqqqq] (-2.3,-3.8187476445545525) node {$B_{\sigma(2)}$};
\draw [fill=ccqqqq] (2.0,4.0) circle (2pt);
\draw[color=ccqqqq] (1.6,4.178768999954727) node {$B_{\sigma(3)}$};
\draw [fill=ffqqqq] (1.0,2.0) circle (3pt);
\draw[color=ffqqqq] (0.9,2.2352017365928663) node {$S$};
\draw [fill=black] (4.0,-0.0) circle (2pt);
\draw[color=black] (4.0,0.22597341703634755) node {$C_{1, \sigma}$};
\draw [fill=black] (1.6,1.6) circle (2pt);
\draw[color=black] (1.7212741311467814,1.28) node {$C_{3, \sigma}$};
\draw [fill=black] (1.4285714285714286,1.7142857142857142) circle (2pt);
\draw[color=black] (1.35,2) node {$C_{2, \sigma}$};
\end{scriptsize}
\end{tikzpicture}
	\caption{}
	\label{pekowe}
\end{figure}

Since on each of the initial lines $L_A$ and $L_B$ there are four distinguished points: the three arrangement points and the point $S$, we can compute their cross-ratio and this is the invariant we want now to bring in the game. Moreover, if the Pappus line $L_C$ passes through $S$, then one can naturally compute the cross-ratio of the four distinguished points on this line.

Let us firstly remind the how this invariant of projective transformations is computed
and what basic properties it enjoys.
%  Generally well known fact is that each point $X$ of the line $L(A,B)$ for some $A,B \in \mathbb{P}^2$ can be written as $$X= \alpha_{X} \cdot A + \mu_{X} \cdot B,$$ where $\alpha_{X}$ and $\mu_{X}$ are some scalars in ambient field (in our case $\mathbb{C}$). The numbers $\alpha_{X}$ and $\mu_{X}$ are coordinates of the point $X$ in coordinate system  $A$ and $B$. Of course any two distinct points of the line $L(A,B)$ might be taken as the coordinate system (projectively these coordinates are defined up to an overall scale factor). For  an ordered quadruple of collinear points $A, B, C$ and $D$ we define the cross ratio as follows
%  $$[A,B;C,D]:= \frac{(\alpha_{A} \cdot \mu_{C}- \alpha_{C} \cdot \mu_{A})(\alpha_{B} \cdot \mu_{D}- \alpha_{D} \cdot \mu_{B})}{(\alpha_{A} \cdot \mu_{D}- \alpha_{D} \cdot \mu_{A})(\alpha_{B} \cdot \mu_{C}- \alpha_{C} \cdot \mu_{B})}.$$

A convenient way to calculate the cross-ratio of four collinear points $A, B, C, D \in \mathbb{P}^2$ 
is given by the determinantal formula
$$[A,B;C,D]= \frac{[\mathcal{O},A,C] \cdot [\mathcal{O},B,D]}{[\mathcal{O},A,D] \cdot [\mathcal{O},B,C]},$$
where $\mathcal{O} $ is an arbitrary point in $\mathbb{P}^2$ not on the line $L(A,B)$.
Here $[A,B,C]$ denotes the determinant of the $3 \times 3$ matrix, which consecutive columns are the homogeneous coordinates of points $A,B$ and $C$ respectively (see \cite{Gebert}, Lemma $4.6$). 

The cross ratio depends on the order of involved points. In fact there are at most six distinct ratios for any quadruple of collinear points. If 
$[A,B;C,D]=\lambda$, then these six values are: 
\begin{equation} \label{ratios}
\lambda, \hspace{0.5cm} \frac{1}{\lambda}, \hspace{0.5cm} 1- \lambda, \hspace{0.5cm}  \frac{1}{1-\lambda}, \hspace{0.5cm} \frac{\lambda}{\lambda -1}, \hspace{0.5cm}  \frac{\lambda-1}{\lambda} \end{equation}
(for details see \cite{Gebert}, Chapter $4$). 

If the value of the cross ratio of four collinear points belongs to the set 
$\left\{-1, \frac12, 2\right\}$ we say that these points form a \textit{harmonic quadruple}.

There is interesting relation between the value of cross ratio for quadruples of points $\{A_1, A_2, A_3, S \}$ and  $\{B_1, B_2, B_3, S \}$ and the condition that a related Pappus line passes through the point $S$.

\begin{theorem}
The Pappus line $L_{C,\sigma}$ passes through the point $S$ if and only if $$[A_1,A_2,A_3,S]=[B_{\sigma(1)},B_{\sigma(2)},B_{\sigma(3)},S].$$
\end{theorem}
\begin{proof}
One can easily check that for coordinates assumed in the proof of Theorem \ref{dualthm}
we have $[A_1,A_2,A_3,S]= 1-a$. The values of the corresponding cross-ratios $[B_{\sigma(1)},B_{\sigma(2)},B_{\sigma(3)},S]$ for $\sigma \in S_3$ are presented in Table \ref{concurrent}.

\begin{table}[ht] \caption{ } \label{concurrent}
\centering
\begin{tabular}{r|c|c|c|c|c|c} 
  $\sigma$ & $id$ & $\tau_2$  & $\tau_2 \tau_3$ & $\tau_3$ & $\tau_3^2$ & $\tau_2 \tau_3^2$ \\
  \hline
  $[B_{\sigma(1)},B_{\sigma(2)},B_{\sigma(3)},S]$ & $\frac{b-1}{b}$ & $\frac{b}{b-1}$ & $1-b$ & $b$ &  $\frac{1}{1-b}$ & $\frac{1}{b}$\\
\end{tabular} 
\end{table}

\vspace{1cm}
\noindent Evaluating the equality between each value of  $[B_{\sigma(1)},B_{\sigma(2)},B_{\sigma(3)},S]$ for $\sigma \in S_3$ with $1-a$ we obtain the same algebraic condition as 
checking if the  $S=(1:1:0)$ belongs to lines $L_{C, \sigma}$, namely 
\begin{equation} \label{glue}
\left\{ \begin{array}{lll}
\vspace{0.3cm}
ab-1=0, & \textrm{if $\sigma= id$}\\ 
\vspace{0.3cm}
ab-a+1, & \textrm{if $\sigma= \tau_2$}\\
\vspace{0.3cm}
a-b=0, & \textrm{if $\sigma= \tau_3^2$}\\
\vspace{0.3cm}
a+b-1=0, & \textrm{if $\sigma= \tau_3$}\\
\vspace{0.3cm}
ab-a-b=0, & \textrm{if $\sigma= \tau_2 \tau_3^2$}\\
\vspace{0.3cm}
ab-b+1=0, & \textrm{if $\sigma= \tau_2 \tau_3$} \\
\end{array} \right.,\end{equation}
which establishes the assertion.
\end{proof}

Assuming that the line $L_{C, \sigma}$ passes through the point $S$ we compute the cross ratio of points $C_{1,\sigma},C_{2, \sigma},C_{3, \sigma}$ and $S$.

\begin{lemma}
   If $S \in L_{C, \sigma}$, for a $\sigma \in S_3$.
   then $$[C_{1,\sigma},C_{2, \sigma},C_{3, \sigma},S]=1-a.$$
\end{lemma}

  It might come as odd, that the cross-ratio of points on the line $L_{C,\sigma}$
  depends only on $a$, whereas their coordinates depend also on $b$, but in fact
  the dependence on $b$ is encoded in conditions \eqref{glue}.

Another natural question here is  whether it is possible that more than one of the lines  $L_{C, \sigma}$ passes through the point $S$. 
It turns out that it is indeed possible, only if points on the initial lines form a harmonic quadrangle.
The more precise coupling between specific pairs of Pappus lines is presented
in Corollary \ref{meet}.

\begin{corollary} \label{meet}
   The following table gives values of $a$ and $b$ for which a pair of Pappus lines passes through $S$ and indicates labels of the lines in the pair.
%\begin{itemize}
%    \item $(a,b)= (2, \frac{1}{2})$ iff $S \in L_{C, id} \cap L_{C, \tau_2}$,
%    \item $(a,b)= (-1, -1)$ iff $S \in L_{C, id} \cap L_{C, \tau_2 \tau_3}$,
%    \item $(a,b)= (\frac{1}{2}, 2)$ iff $S \in L_{C, id} \cap L_{C, \tau_2 \tau_3^2}$,
%    \item $(a,b)= (-1, 2)$ iff $S \in L_{C, \tau_2} \cap L_{C, \tau_3}$,    
%    \item $(a,b)= (\frac{1}{2}, -1)$ iff $S \in L_{C, \tau_3^2} \cap L_{C, \tau_2}$,
%    \item $(a,b)= (\frac{1}{2}, \frac{1}{2})$ iff $S \in L_{C, \tau_3} \cap L_{C, \tau_2 \tau_3}$,   
%    \item $(a,b)= (2, 2)$ iff $S \in L_{C, \tau_3^2} \cap L_{C, \tau_2 \tau_3}$,    
%    \item   $(a,b)= (2, -1)$ iff $S \in L_{C, \tau_3} \cap L_{C, \tau_2 \tau_3^2}$,    
%    \item $(a,b)= (-1, \frac{1}{2})$ iff  $S \in L_{C, \tau_3^2} \cap L_{C, \tau_2 \tau_3^2}$.   
%\end{itemize}
\begin{table}[ht] \caption{ } \label{double}
\centering
\renewcommand*{\arraystretch}{1.4}
\begin{tabular}{c|c|c|c} 
$b$ \slash $a$ & $-1$ & $\frac12$  &  $2$\\
\hline
$-1$ & $L_{C, id} \cap L_{C, \tau_2\tau_3}$ & $L_{C, \tau_3^2} \cap L_{C, \tau_2}$  & $L_{C, \tau_3} \cap L_{C, \tau_2 \tau_3}$ \\
   \hline
$\frac12$ & $L_{C, \tau_3^2} \cap L_{C, \tau_2 \tau_3^2}$ & $L_{C, \tau_3} \cap L_{C, \tau_2 \tau_3}$  & $L_{C, id} \cap L_{C, \tau_2}$ \\
   \hline
$2$ & $L_{C, \tau_3} \cap L_{C, \tau_2}$ & $L_{C, id} \cap L_{C, \tau_2 \tau_3^2}$ & $L_{C, \tau_3^2} \cap L_{C, \tau_2 \tau_3}$ \\
\end{tabular} 
\end{table}
\end{corollary}
Note that there is certain regularity among the entries in Table \ref{double}. Namely passing from the entry with $$L_{C,\sigma_1}\cap L_{C,\sigma_2}$$ 
to the right requires to multiply both index permutations $\sigma_1, \sigma_2$ by $\tau_3^2$, so that the entry to the right is
$$L_{C,\sigma_1\tau_3^2}\cap L_{C,\sigma_2\tau_3^2},$$
whereas passing from the same entry downwards requires to multiply the first coordinate by $\tau_3^2$ and the second by $\tau_3$, so that the entry below is
$$L_{C,\sigma_1\tau_3^2}\cap L_{C,\sigma_2\tau_3}.$$

\begin{remark}
It is not possible that more than two of the lines $L_{C, \sigma}$ pass through the point $S$ (assuming that all lines $L_{C, \sigma}$ are mutually distinct). 
This follows immediately from Corollary \ref{meet}
but it can be seen also directly from conditions in \eqnref{glue}. Indeed, taking any three of equations there gives either a direct contradiction, or one gets the condition
\begin{equation} \label{overlap}
    b^2-b+1=0 
\end{equation}
for the parameter $b$. However, if \eqnref{overlap} is satisfied, then $3$ Pappus lines overlap and there are no $3$ distinct lines passing through $S$. The overlapping triples of lines are either indexed by odd permutations $L_{C, \tau_2}, L_{C, \tau_2 \tau_3}, L_{C, \tau_2 \tau_3^2}$ or by the even ones $L_{C, id}, L_{C, \tau_3}, L_{C,  \tau_3^2}$ depending on which root of equation \eqnref{overlap}
is taken for the value of $b$.
\end{remark}

\section{The super Pappus  arrangement} \label{super}

Theorem \ref{dualthm} says that the points $M^{*}_{C, \sigma}$ are contained in lines $L_A$ and $L_B$
but they differ typically from the initial points $A_i$ and $B_i$ for $i=1,2,3$. A natural question here is whether it may happen that these sets points are equal. It turns out that there do
exist values of $a$ and $b$ for which the returning points $M^{*}_{C, \sigma}$ are exactly the initial points. When this happens, we speak of a \textit{super Pappus arrangement}. Oddly, this is related to the problem of Pappus lines passing through the intersection point $S=L_A\cap L_B$ of the initial lines discussed in Section \ref{sec: additional incidences}. This is made precise in the following statement.

\begin{theorem}
The following conditions are equivalent
\begin{itemize}
    \item[$i)$] Two triples $A_i$ and $B_i$ for $i=1,2,3$ of mutually distinct collinear points form a super Pappus arrangement;
    \item[$ii)$] There is a pair of Pappus lines $L_{C, \sigma}$ passing through the point $S= L_A \cap L_B$;
    \item[$iii)$] Each quadruple of points $\{ A_1, A_2, A_3,S \}$, $\{ B_1, B_2, B_3,S \}$ and $\{ C_{1,\sigma}, C_{2,\sigma}, C_{3,\sigma},S \}$ is a harmonic quadruple.
\end{itemize}
\end{theorem}

\begin{proof}
We will proof these equivalences by showing that each condition leads to the same values of the parameters $a$ and $b$. Firstly we take a closer look at the condition $i)$. By Theorem \ref{dualthm} it is enough to establish when the following sets coincide
$$\begin{array}{lcl}
\vspace{0,5cm}
\{ M_{C, id}^{*}, \hspace{0.2cm} M_{C, \tau_3}^{*}, \hspace{0.2cm} M_{C, \tau_3^2}^{*} \} & = & \{B_1, B_2, B_3 \},\\
\{ M_{C, \tau_2}^{*}, \hspace{0.2cm}  M_{C, \tau_2 \tau_3}^{*},  \hspace{0.2cm} M_{C, \tau_2 \tau_3^2}^{*} \}   & = & \{A_1, A_2, A_3 \}.\\
\end{array}$$
Direct computation shows that we have

\vspace{1cm}

\begin{table}[ht] 
\centering
\begin{tabular}{l|l|l} 
$a=2$ & $a=-1$  &  $a=\frac{1}{2}$\\
\hline
 $A_1 = M^{*}_{ \tau_2}$ &  $A_1 = M^{*}_{ \tau_2 \tau_3^2}$  & $A_1 = M^{*}_{ \tau_2 \tau_3}$\\

 $A_2 = M^{*}_{ \tau_2 \tau_3}$ &  $A_2 = M^{*}_{ \tau_2 }$  &$A_2 = M^{*}_{ \tau_2 \tau_3^2}$ \\

  $A_3 = M^{*}_{ \tau_2 \tau_3^2}$ &  $A_3 = M^{*}_{ \tau_2 \tau_3}$ & $A_3 = M^{*}_{ \tau_2}$
\end{tabular} 
\end{table}

\noindent and similarly

\vspace{0.6cm}
\begin{table}[ht] 
\centering
\begin{tabular}{l|l|l} 
$b=2$ & $b=-1$  &  $b=\frac{1}{2}$\\
\hline
 $B_1 = M^{*}_{ \tau_3^2}$ &  $B_1 = M^{*}_{ id}$  & $B_1 = M^{*}_{ \tau_3}$\\

 $B_2 = M^{*}_{id}$ &  $B_2 = M^{*}_{ \tau_3 }$  &$B_2 = M^{*}_{ \tau_3^2}$ \\

  $B_3 = M^{*}_{ \tau_3}$ &  $B_3 = M^{*}_{ \tau_3^2}$ & $B_3 = M^{*}_{id}$ 
\end{tabular} .
\end{table}

Thus the configuration is the super Pappus arrangement, if and only if, the pair $(a,b)$ is such that $a, b \in \{2, -1, \frac{1}{2} \}$. Then by Corollary \ref{meet} we conclude, that the conditions $i)$ and $ii)$ are equivalent. 

We complete the proof showing, that for any $a, b \in \{2, -1, \frac{1}{2} \}$ the sets $\{ A_1, A_2, A_3,S \}$, $\{ B_1, B_2, B_3,S \}$ and $\{ C_{1,\sigma}, C_{2,\sigma}, C_{3,\sigma},S \}$ are the harmonic quadruples. Indeed, let us notice that equalities $$[A_1,A_2,A_3,S]=[C_{1,\sigma},C_{2, \sigma},C_{3, \sigma},S]=-1$$
imply $a=2$. Taking under consideration possible orders of points $A_{i}$ and $C_{i, \sigma}$, by \eqnref{ratios} we obtain two additional values of $a$, namely $-1$ and $\frac{1}{2}$.

Analogously $[B_{\sigma(1)},B_{\sigma(2)},B_{\sigma(3)},S]=-1$, if and only if $b \in \{2, -1, \frac{1}{2} \}$ (compare with ratios in the Table \ref{concurrent}), what ends the proof.
\end{proof}
   Our research shows that, even though the Pappus arrangement is known for ages, 
   it is an inspiring source of ideas in the contemporary mathematics.

\paragraph*{\emph{Acknowledgement.}}
   We would like to thank Piotr Pokora and Tomasz Szemberg for helpful remarks.
   We are also obliged to Grzegorz Malara for assistance with programming some symbolic computations.
%*****************************************************************************

%***************************************************************************** % Addresses

\bigskip \small

\bigskip
   Magdalena~Lampa-Baczy\'nska, Daniel W\'ojcik,
   Institute of Mathematics
   Pedagogical University of Cracow,
   Podchor\c a\.zych 2,
   PL-30-084 Krak\'ow, Poland

\nopagebreak

  \textit{E-mail address:} \texttt{lampa.baczynska@wp.pl}

%   \textit{E-mail address:} \texttt{szemberg@up.krakow.pl}

   \textit{E-mail address:} \texttt{daniel.wojcik@up.krakow.pl}

\bigskip

\end{document}